\documentclass{article}
\usepackage[utf8]{inputenc}
\usepackage{amsmath}
\usepackage{amssymb}
\usepackage{amsfonts}
\usepackage{amsthm}
\usepackage{comment}
\usepackage{stmaryrd}
\usepackage{url}
\usepackage[table]{xcolor}
\newcommand{\N}{\mathbb{N}}
\newcommand{\Z}{\mathbb{Z}}
\newcommand{\Q}{\mathbb{Q}}

\newcommand{\C}{\mathbb{C}}
\newcommand{\F}{\mathbb{F}}
\newcommand{\E}{E}
\newcommand{\f}{F}
\usepackage{bbm}
\newtheorem{theorem}{Theorem}[section]
\newtheorem{lemma}[theorem]{Lemma}
\newtheorem{corollary}[theorem]{Corollary}
\newtheorem{proposition}[theorem]{Proposition}
\newtheorem{conjecture}[theorem]{Conjecture}
\newtheorem{remark}[theorem]{Remark}

\DeclareMathOperator{\Sh}{Sh}

\title{On the Elliptic Curve $X_0(49)$ over Quadratic Extensions}
\author{Charlotte Dombrowsky}
\date{}

\begin{document}
\maketitle
\begin{abstract}
    We study the rank of the modular curve $X_0(49)$ over quadratic extensions. Assuming the Birch and Swinnerton-Dyer conjecture, we show that the rank over $\Q(\sqrt{d})$ is positive if and only if the number of solutions of two explicit ternary quadratic forms is the same. Following the approach of Tunnell, we apply a theorem due to Waldpurger which relates twisted $L$-functions of integer weight modular forms to coefficients of half-integral weight modular forms. To find suitable functions of half-integral weight, we use a  decomposition described by Ueda. 
\end{abstract}
\section{Introduction}

A well-known problem in number theory is the congruent number problem, which aims to determine if an integer $d$ is the area of a right triangle with sides given by rational numbers. Moreover, the problem is equivalent to finding non-torsion points on a family of quadratic twists of elliptic curves. In \cite{Tunnell}, Tunnell has found a beautiful criterion for the congruent number problem involving the number of solutions of certain quadratic ternary forms over the integers.
Tunnell’s criterion is based on the work of Waldspurger \cite{Waldspurger} on quadratic twist of L-functions.

Purkait applied this method to several other elliptic curves \cite{Purkait_2013}, while in \cite{Jones_Rouse}, Jones and Rouse apply Tunnell's method to find solutions to the cubic Fermat equation over quadratic number fields.

In this paper, we apply Tunnell's approach to the modular curve $X_0(49)$, which is an elliptic curve with model given by the Weierstrass equation
\begin{align*}
    \E: y^2+xy=x^3-x^2-2x-1.
\end{align*}
Using the result of Waldspurger \cite{Waldspurger}, we investigate for which integer $d$, the curve $E$ has infinitely many points over the quadratic extension $\Q(\sqrt{d})$. This allows us to find easily computable criterions for the existence of non-torsion points using integer solutions of quadratic forms.

The main result of the paper is the following:

\begin{theorem}\label{Main-Theorem}
 Let $d$ be a positive, squarefree integer coprime to $7$. Assuming the BSD conjecture (\ref{BSD}) we have:
    \begin{itemize}
        \item[(i)] If $\left( \frac{d}{7}\right)=1$, then $\E(\Q(\sqrt{d}))$ has non torsion points if and only if 
        \begin{align*}
            \#\{ &(x,y,z): \quad 28x^2+y^2+14z^2+14xz=d \} \\
        =\#\{ &(x,y,z): \quad 49x^2+2y^2+4z^2+2yz=d \}.
        \end{align*}
        \item[(ii)] If $\left( \frac{d}{7}\right)=-1$ and $d \not\equiv 5 \mod 8$, then $\E(\Q(\sqrt{d}))$, has non torsion points if and only if 
    \begin{align*}
            \#\{ &(x,y,z): \quad 12x^2+3y^2+12z^2+2xy+10xz+2yz=d \} \\
        =\#\{ &(x,y,z): \quad 17x^2+5y^2+5z^2+2xy-2xz+4yz=d \}.
    \end{align*}
    \item[(iii)] If $\left( \frac{d}{7}\right)=-1$ and $d \equiv 1 \mod 4$ then $\E(\Q(\sqrt{d}))$, has non torsion points if and only if 
        \begin{align*}
            \#\{ &(x,y,z): \quad 13x^2+12y^2+12z^2+8xy - 8 xz -4 yz=d \} \\
        =\#\{ &(x,y,z): \quad 17x^2+5y^2+17z^2+2xy +6 xz + 2yz=d \}.
        \end{align*}
    \end{itemize} 
\end{theorem}
We note that when $d$ is a negative squarefree integer coprime to $7$, then $\E(\Q(\sqrt{d}))$ has positive rank.
Moreover, for any $d$ squarefree integer coprime to $7$, $\E(\Q(\sqrt{d}))$ has non torsion points if and only if $\E(\Q(\sqrt{-7d}))$ has non torsion points. Together with the theorem this covers all possibilities for $d$.

\medskip

In \cite{Pacetti_2007}, Pacetti and Tornaria relate the $L$-function of $X_0(49)$ to an explicitly constructed half-integral weight modular form. However, they do not give an explicit criterion for the coefficients of the latter one.

Moreover, Ehlen, Guerzhoy, Kane and Rolen relate in \cite{Ehlen_2020} the $L$-function of quadratic twists of elliptic curves to local polynomials. They show that the $L$-function is zero if and only if the value of a local polynomial at two points is the same. Note that their result is only applicable to discriminants $d$ such that $\left( \frac{-7}{|d|}\right)=-1$, while our result does not have such a restriction.

\medskip

To prove Theorem \ref{Main-Theorem}, we study quadratic twists of the elliptic curve and the result of Waldspurger (Theorem \ref{Waldspurger}), which relates the $L$-function of quadratic twists of an integer weight modular form to coefficients of certain half-integral weight modular forms. After finding suitable functions of half-integral weight, we express their coefficients using certain theta-series. 

\medskip

The paper is structured as follows: In Section $2$, we provide some background on the elliptic curve $\E$ and its twists.
In Section $3$, we present Waldspurger's Theorem and compute different modular forms of half-integral weight. Finally in Section $4$, we prove the main theorem.
\section{Preliminaries}
We study the elliptic curve $\E_1$ with model defined by the minimal Weierstrass equation
\begin{align*}
    E_1: \quad y^2+xy=x^3-x^2-2x-1,
\end{align*}
together with its quadratic twists. For any squarefree number $d$, we denote by 
\begin{align*}
    \E_d: \quad dy^2=x^3-35x-98
\end{align*}
the quadratic twist of $E_1$ over $\Q(\sqrt{d})$. The curve $E_1$ has complex multiplication (CM) by $\Z[\frac{1+\sqrt{-7}}{2}]$.
By reducing $E_1$ over $\F_3$ and $\F_5$, one can show that the torsion group is isomorphic to $\Z/2\Z$. As we show in Lemma \ref{torsion-E_d(Q)}, the torsion group of $E_d(\Q)$ is also isomorphic to $\Z/2\Z$.

We recall that that the $L$-function of an elliptic curve is:
\begin{align*}
    L(E,s)
    =\sum_{n=1}^{\infty} a_n(E)n^{-s}.
\end{align*}
where for $p$ prime $a_p$ is the trace of the Frobenius. Since $E_1$ has CM over $\Q(\sqrt{-7})$, we have for any prime $p \geq 5$ of good reduction that:
    \begin{align}\label{Frobenius_CM_-1}
        a_p:=p+1-|E_1(\F_p)|=0 \Leftrightarrow \left( \frac{-7}{p}\right)=-1.
    \end{align}

The Birch and Swinnerton-Dyer Conjecture relates the $L$-function to the rank of an elliptic curve:
\begin{conjecture}[Birch and Swinnerton-Dyer]\label{BSD}
    For any elliptic curve $E$ defined over $\Q$, the rank of $E$ equals the order of $L(E,s)$ at $s=1$.
    
    In particular, it is expected that the rank of $E$ is non-zero, if and only if $L(E,1)=0$.
\end{conjecture}
We have the following result proved by Kolyvagin (\cite{kolyvagin}), Gross and Zagier (\cite{gross-zagier}):
\begin{theorem}[Gross-Zagier, Kolyvagin]
    Let $E$ be an elliptic curve over $\Q$. Then we have
    \begin{align*}
        L(E,1) \neq 0 &\Rightarrow \text{rank }E(\Q) = 0,\\
        L(E,1) = 0 \text{ and } L'(E,1)\neq 0 &\Rightarrow \text{rank }E(\Q)=1.
    \end{align*}
\end{theorem}
For the elliptic curve $E_1$, one can compute that:
\begin{align*}
L(E_1,1) \approx 0.9666558528.
\end{align*}
Therefore the rank is zero and $E_1(\Q)$ is equal to its torsion group, i.e.
\begin{align}\label{eq: E over Q}
    E_1(\Q)\cong \Z /2\Z.
\end{align}
For the rest of the paper, we will assume the BSD Conjecture. 
\smallskip

We prove in Lemma \ref{torsion-correspondance} that $E_1(\Q(\sqrt{d}))$ is of positive rank if and only if $E_d(\Q)$ is of positive rank. Thus we can study the $L$-function of the twisted elliptic curve $E_d$ over the rational numbers in order to deduce for which quadratic extensions the rank is non-zero.

 The modularity theorem associates to every elliptic curve of conductor $N$ a modular form of level $N$ and weight $2$ with the same $L$-function. In the following, we describe relations between twists of modular forms and $L$-functions.

 For any character $\chi$ and any modular function $F=\sum_{n=0}^{\infty} a_n q^n$, the twist of $F$ by $\chi$ is given by
 \begin{align*}
     F\otimes \chi:= \sum_{n=0}^{\infty} a_n \chi(n) q^n
 \end{align*}
 
 We denote by $\chi_d$ the Dirichlet character associated to the quadratic field $\Q(\sqrt{d})$, i.e. $\chi_d(p)=\left( \frac{d}{p}\right)$
    where the $\left( \frac{\cdot}{\cdot}\right)$ denotes the Kronecker symbol. 
For any elliptic curve $E$ with associated modular form $F=\sum_{n=0}^{\infty} a_n q^n$, the form associated to its $d$th quadratic twist $E_d$ is given by
\begin{align}\label{eq: modular form of E_d}
    F_d=\sum_{n=0}^{\infty} a_n \left(\frac{d}{n}  \right) q^n.
\end{align}

We now prove some results about the relation between the $L$-function of a modular form and the $L$-function of its (quadratic) twist. In particular, we show that for certain twists, the $L$-function of an eigenform and its twist are essentially the same when evaluated at $1$.
\begin{lemma}\label{prop-twist-L-function}
    Let $F \in M_k(N,\chi)$ be a normalized eigenform and let $\chi_0$ be a Dirichlet character modulo $M$. We denote the twist of $F$ by $\chi_0$ with $F_{\chi_0}$. Then $F_{\chi_0}$ is a normalized eigenform. If $F$ is a cusp form, then $F_{\chi_0}$ is also a cusp form.
\end{lemma}
\begin{proof} 
Let
    \begin{align*}
        F(z)=\sum_{n=0}^{\infty} a_n q^n \quad \text{ and }
        F_{\chi_0}(z)=\sum_{n=0}^{\infty} b_n q^n
        =\sum_{n=0}^{\infty} a_n \chi_0(n)q^n.
    \end{align*}
    Proposition 2.8 in \cite{Ono} implies that $F_{\chi_0} \in M_k(NM^2, \chi\chi_0^2)$ and that $F_{\chi_0}$ is a cusp form if $F$ is. Proposition 5.8.5 in \cite{Diamond-Shurman} states that being a normalized eigenform is equivalent to three conditions on the Fourier coefficients. More explicitly $F$ satisfies the following properties:
    \begin{itemize}
        \item[(i)] $a_1=1$
        \item[(ii)] $a_{p^r}=a_p a_{p^{r-1}} - \chi(p) p^{k-1} a_{p^{r-2}}$ for all $p$ prime and $r \geq 2$,
        \item[(iii)] $a_{mn}=a_m a_n$ when $(m,n)=1$.
    \end{itemize}
    One easily checks that $b_1=a_1\chi_0(1)=1$.
    Using the multiplicativity of the $\chi_0$, we show that the coefficients of $F_{\chi_0}$ satisfy conditions $(ii)$ and $(iii)$ as well:
    \begin{itemize}
        \item[(i)]Let $r \geq 2$. Then:
        \begin{align*}
            b_{p^r}
            &=\chi_0(p^r) a_{p^r}\\
            &=\chi_0(p)a_p\chi_0(p^{r-1})a_{p^{r-1}} - \chi(p)\chi_0^2(p) p^{k-1} \chi(p^{r-2})a_{p^{r-2}} \\
            &=b_p b_{p^{r-1}} - \chi(p) \chi_0^2(p) p^{k-1} b_{p^{r-2}}.
        \end{align*}
        \item[(ii)] For $(n,m)=1$ we have $b_{nm}=\chi_0(mn)a_{nm}=\chi_0(m)\chi_0(n)a_n a_m=b_n b_m$.
    \end{itemize}
    Thus Proposition 5.8.5 in \cite{Diamond-Shurman} implies that $F_{\chi_0}$ is a normalized eigenform.
\end{proof}
This Proposition allows us to compare the $L$-functions of cusp functions and certain twists. More specifically, we have:
\begin{corollary}\label{L-function-p-twist-1}
 Let $F=\sum_{n=1}^{\infty} a_n q^n \in S_k(N,\chi)$ be a normalized eigenform. Fix a prime $p$. Denote the twist by $\chi_{p^2}$ by $F_{p^2}:=F \otimes \chi_{p^2}$.
 \begin{itemize}
    \item[(i)]Then $L(F_{p^2},s)=L(F,s) (1-a_p p^{-s} + \chi(p) p^{k-1-2s}).$
    \item[(ii)] If $a_pp^{-k/2+1}-\chi(p) \neq p$ then $L(F_{p^2},k/2)=0 \Leftrightarrow L(F,k/2)=0$.
\end{itemize}
\end{corollary}

\begin{proof}
Let $F_{p^2}(z)=\sum_{n=0}^{\infty} b_n q^n$. By Theorem 5.9.2 in \cite{Diamond-Shurman}, the $L$-functions of $F$ and $F_{p^2}$ are given by:
    \begin{align*}
        L(F,s)&=\prod_{l: \text{prime}}(1-a_l l^{-s}+\chi(l)l^{k-1-2s})^{-1} \\
        L(F_{p^2},s)&=\prod_{l: \text{prime}}(1-b_l l^{-s}+\chi(l)\chi_{p^2}^2(l)l^{k-1-2s})^{-1}
    \end{align*}
    The factors on the right hand side only differ  when $l=p$. The first one is $(1-a_pp^{-s}+\chi(p)p^{k-1-2s})^{-1}$, while the latter one is $1$. This implies the first claim.
    For the second claim, plugging in $s=k/2$, yields:
\begin{align*}
    L(F_{p^2},k/2)=L(F,k/2) (1-(a_pp^{-k/2}-\chi(p)p^{-1})).
\end{align*}
\end{proof}
\begin{remark}
If $k=2$, then the second part of the corollary implies that $E$ has non-torsion points over $\Q$ of and only if its $p^2$th twist has non-torsion points over $\Q$. However, this can also be seen from the following bijection:
\begin{align*}
    \{ (x,y) :   p^2y^2= x^3+Ax+B \} &\rightarrow \{ (x,y) :  y= x^3+Ax+B \}\\
    (x,y) &\mapsto (x,py)\\
    (x,y/p) &\mapsfrom (x,y)
\end{align*}
\end{remark}

We recall that the completion $\Lambda(F,s)$ of the $L$-function associated to a new form $F \in S_k^{new}(\Gamma_0(N))$ satisfies the following functional equation:
 \begin{equation*}
        \Lambda(F,s)=\epsilon \Lambda(F,k-s).
    \end{equation*}
The number $\epsilon$ is called the sign of the functional equation and is given by $(-1)^kw_N$,  where $w_N$ is the eigenvalue of $F$ under the Fricke involution $W_N$. If $\epsilon=-1$ and $k=2$, the $L$-function evaluated at $1$ is equal to $0$.
Moreover, the functional equation of its twist is given by (see also Lemma 9.2 in \cite{Ono})
\begin{equation}\label{functional-equation-l-function}
    \Lambda\left(F_d ,s \right)=\left( \frac{d}{-N}\right)\epsilon \Lambda\left(F_d, k-s \right).
\end{equation}

Finally, we recall that one can create an old form from a new form in the following way: For any $F= \sum_{n \geq 0} a(n)q^n \in M_{k}(N,\chi )$ and any positive integers $d$ such that $d|N$, we can set $F_{old}:=\sum_{n \geq 0} a_{dn}q^n$. Then $F_{old}(z) \in M_{k}(N, \chi)$. If $F(z)$ is a cusp form, then so is $F_{old}(z)$. Moreover $F_{old}$ has the same system of Hecke eigenvalues as $F$ for all primes $p \nmid N$. Indeed let $\mathbbm{1}_n(p)=1$ if $p$ divides $n$, and $0$ otherwise. We compute:
\begin{equation}\label{old-same-heigenvalues}
    \begin{split}
    a_n(T_pF_{old})&=\sum_{t | (n,p)} \chi(t) t\cdot a_{np/t^2}(F_{old})\\
    &=\chi(1)\cdot a_{np}(F_{old})+
    \mathbbm{1}_n(p)\chi(p)p\cdot a_{n/p}(F_{old})\\
    &=a_{dnp}(\f)+\mathbbm{1}_n(p)p \cdot a_{dn/p}(\f)\\
    &= a_{dn}(T_p\f)
    =a_p(\f)a_{dn}(\f)
    =a_p(\f)a_n(F_{old}).
    \end{split}
\end{equation}

\section{Modular Forms of Half-Integral Weight}
\subsection{Waldspurger's Theorem}
We study the $L$-function by using a theorem due to Waldspurger. Essentially, it states the following: given a newform of integer weight and a form of half-integral weight such that they have almost the same system of Hecke eigenvalues, we can relate the $L$-function of (certain) twists of the integer weight modular form evaluated at the central value to the coefficients of the half-integral weight modular form.

\begin{theorem}[Waldspurger]\label{Waldspurger}
Let $N_1$ be a number divisible by $4$ and $\chi$ be a Dirichlet character modulo $N_1$. Let $N_2$ be an integer such that $\chi^2$ is defined modulo $N_2$. Let $F \in M_{k}( N_2, \chi^2)$ be a newform. Assume that for all primes $p$, there exists $a_p \in \C$ with $T(p)F=a_p F$. Let $f=\sum_{n=1}^{\infty} c_n q^n \in S_{(k+1)/2}(N_1, \chi)$ be a modular form such that
\begin{align*}
    T(p^2)f=a_pf, \quad \text{ for almost all }p \nmid N_1.
\end{align*}
    Set $\chi_0(n):=\chi(n) \left( \frac{-1}{n}\right)^{k/2}$.
    Then for any squarefree integers $n_1$ and $n_2$ such that for all $p \mid N$, we have $n_1/n_2 \in (\Q_p^*)^2$:
    \begin{align*}
        c_{n_1}^2L(F\otimes \chi_0^{-1}\chi_{n_2},k/2) \chi(n_2/n_1)n_2^{(k-1)/2} = c_{n_2}^2L(F\otimes \chi_0^{-1}\chi_{n_1},k/2) n_1^{(k-1)/2}.
    \end{align*}
\end{theorem}
We want to apply this theorem to the modular function associated to the elliptic curve $\E$. We use different methods to find half-integral weight modular form with (almost) the same system of Hecke-eigenvalues. First, we compute a subspace of half-integral weight forms described by Ueda \cite{Ueda}. Then we compute the Shimura lift (which preserves the system of Hecke eigenvalues) of several functions.

\subsection{Ueda's Decomposition}\label{Uedas-Decomposition}
In \cite{Ueda}, Ueda gives an isomorphism between a subspace of half-integral weight modular forms and a direct sum of subspaces of integer weight modular forms. This isomorphism preserves the action of the Hecke operators.

\medskip

We decompose the space of integer weight modular forms into subspaces:  Let $S_2(p^2)^{\perp}$ the orthogonal complement of $S_2(p)^{new}\otimes \rho_p + S_2(1)^{new}\otimes \rho_p$ in $S_2(p^2)^{new}$ with respect to the Petersson inner product, where $\rho_p$ is the quadratic character given by $\rho_p=\left( \frac{\cdot}{p}\right)$. Let $W_{p^2}$ be the Fricke operator on $S_k(p^2)$, i.e  $(f|W_{p^2})(z)=z^{-k} f(-1/(p^2z))$. For $\epsilon, \epsilon' =\pm$, we set:
\begin{align*}
 S_2(p^2)^{\epsilon,\epsilon'}:=\{F \in S_2(p^2)^{\perp}: \quad F|W_{p^2}=\epsilon F, \quad F \otimes \rho_p|W_{p^2}= \epsilon' F \otimes \rho_p\}.  
\end{align*}

\medskip

We also decompose the space of half-integral weight forms in subspaces: Recall that the Kohnen plus space which is given by
\begin{equation}\label{Def-kohnen-space}
\begin{split}
        S_{k+1/2}^+(4N,\chi):&=\\
    &\left\{ f= \sum_{n=1}^{\infty} a(n)q^n \in S_{k+1/2}(4N, \chi) :
    a(n)=0, \text{ if } n \equiv 2, (-1)^{k+1}\epsilon \bmod 4 \right\}
    \end{split}
\end{equation}
Here $N$ is an odd number and $N_0 | 4N$ such that $\left( \frac{N_0}{\cdot}\right)=\chi(\cdot)$. Set $\epsilon=\left( \frac{-1}{N_0}\right)$.
We decompose the Kohnen plus space even further.
Denote by $\Omega(4N,\chi)$ the space of pairs $(\phi, t)$ satisfying the following: $\phi$ is a primitive character modulo $r$ and $t$ is a positive integer such that:
\begin{align*}
    \text{(i) }\phi(-1)=-1, \text{ (ii) }4tr^2 | 4N, \text{ and (iii) }\chi=\left(\frac{-t}{\cdot}\right)\phi \text{ as a character mod }N.
\end{align*}

Note that for any quadratic character, we can define a theta series by:
\begin{align*}
    h(\phi, z):=(1/2)\sum_{ m \in \Z} \phi(m)m e^{ 2 \pi m^2 z} \quad \forall z \in \mathcal{H}.
\end{align*}
Let $U_{3/2}^+(4N, \chi)$ be the subspace of the Kohnen space $S_{3/2}^+(4N, \chi)$ generated by $\{ h(\phi, tz): (\phi, t) \in \Omega(4N, \chi) \}$ over $\mathbb{C}$. Finally, we define $V_{3/2}^+(4N, \chi)$ to be the orthogonal complement of $U_{3/2}^+(4N, \chi)_K$ in $S_{3/2}^+(4N,\chi)$ with respect to the Petersson inner product.
 
The following result  is Proposition $3, (4)$ with $k=1$ in \cite{Ueda}.
\begin{proposition}\label{Ueda-proposition}
 We have the following decomposition of spaces of cusp forms as modules over the Hecke algebra:

\begin{align*}
    V_{3/2}^+(4p^2, \chi_{4p}) \cong
    &\left(1+\left( \frac{-1}{p}\right) \right) \left(S_2(p^2)^{+,+} \oplus S_2(p^2)^{-,+}\right)\\
    & \oplus \left(1-\left( \frac{-1}{p}\right) \right) \left(S_2(p^2)^{+,-} \oplus S_2(p^2)^{-,-}\right) \\
    & \oplus \left( S_2(p)^{new} \otimes \chi_{4p} \right) \\
    & \oplus \left(1+ \left( \frac{-1}{p}\right)\right) \left( S_2(1)^{new}\otimes \chi_{4p} \right)\\
    & \oplus 3S_2(p)^{new}\oplus 4S_2(1)^{new}.
\end{align*}
Here the number in front of a module denotes its multiplicity. 
\end{proposition}

We apply Proposition \ref{Ueda-proposition}  for $p=7$. All computations have been done with MAGMA. A basis for $S_{3/2}(4\cdot 7^2, \chi_{28})$ is given by:
\begin{align*}
    h_1=&q + 2q^{11} - q^{15} + q^{16} - 2q^{18} - q^{22} - q^{30} - q^{32} - q^{36} - q^{39} + O(q^{43})\\
    h_2=&q^2 - q^{18} - 2q^{23} + q^{25} - q^{29} + 2q^{36} - q^{37} + O(q^{43})\\
    h_3=&q^3 - q^{10} - q^{12} + q^{20} + q^{26} - q^{27} + q^{33} - q^{38} + q^{40} - q^{41} + O(q^{43})\\
    h_4=&q^4 - q^{11} - q^{15} + q^{18} + 2q^{22} - q^{25} - q^{30} - q^{32} - q^{36} - q^{39} + O(q^{43})\\
    h_5=&q^5 - 2q^{12} - q^{13} - q^{17} + 2q^{20} + 2q^{24} + 2q^{33} - q^{41} + O(q^{43})\\
    h_6=&q^6 - q^{10} + q^{17} + q^{19} - q^{20} - q^{24} - q^{26} - q^{27} + 2q^{31} - q^{33} - q^{34} 
    + q^{40} + O(q^{43})\\
    h_7=&q^7 - q^{14} - q^{28} + O(q^{43})\\
    h_8=&q^8 + q^{11} - q^{15} + q^{16} - q^{18} - q^{29} - q^{30} - q^{37} - q^{39} + 
    O(q^{43})\\
    h_9=&q^9 - q^{15} - q^{16} + q^{22} + q^{25} - q^{30} - q^{32} + q^{36} - q^{39} + O(q^{43})
\end{align*}
Note that the Sturm bound for half-integral weight modular forms (see Appendix \ref{Sturm}), implies that it suffices to compute the first $42$ coefficients to determine any element in $S_{3/2}(4 \cdot 7^2, \chi_{28})$.

 To compute the space $V_{3/2}^+(4\cdot 7^2, \chi_{28})$, we first compute the Kohnen plus subspaces  $S^+_{3/2}(4 \cdot 7^2, \chi_{28})$ and its subspace $U^+_{3/2}(4 \cdot 7^2, \chi_{28})$. 

\begin{lemma} The space Kohnen plus space $ S^+_{3/2}(4 \cdot 7^2, \chi_{28})$ is spanned by
    \begin{align*}
        g_1&:=h_1-2h_8+h_9,\\
        l_1&:=h_4+h_8-2h_9, \quad \text{and}\\
        f_1&:=h_5.
    \end{align*}
\end{lemma}

\begin{proof}
 Since $\chi_{28}=\left( \frac{28}{\cdot}\right)$, we have that $\epsilon=\left(\frac{-1}{28}\right)=-1$. Thus we need to compute
\begin{align*}
          \left\{ f= \sum_{n=1}^{\infty} a(n)q^n \in S_{3/2}(4 \cdot 7^2, \chi_{28}) :
          a(n)=0, \text{ if } n \equiv 2, 3 \quad (\text{mod 4}) \right\}.
\end{align*}
From Sturm's bound for the Kohnen subspace (see Appendix \ref{Sturm}), it follows that it suffices to check for $n \leq 42$, $n \equiv 2,3 \mod 4$ if $c_n=0$. We look at linear combinations of the basis functions. Assume that for some $\alpha_n \in \C$ for $1 \leq n \leq 9$, we have $h=\sum_{n=1}^9 \alpha_n h_n$ is an element of the Kohnen space. As $h_2$ is the only function such that $c_{h_i}(2)\neq 0$, $\alpha_2=0$. By looking at the 3rd, 6th and 14th coefficients respectively, we see that also $\alpha_3=\alpha_6=\alpha_7=0$.
Comparing the other coefficients, we find that
\begin{align*}
    \alpha_8=-2\alpha_1+\alpha_4 \text{ and }\alpha_9=\alpha_1-2\alpha_4.
\end{align*}
Thus the Kohnen subspace is the space given by:
\begin{align*}
    \left\{ \alpha_1h_1+\alpha_4h_4+\alpha_5h_5 +(-2\alpha_1+\alpha_4)h_8+(\alpha_1-2\alpha_4)h_9 : \quad \alpha_1,\alpha_4, \alpha_5 \in \mathbb{C} \right\} \\
    =  \left\{ \alpha_1(h_1-2h_8+h_9) +\alpha_4(h_4+h_8-2h_9)+\alpha_5h_5 : \quad \alpha_1,\alpha_4, \alpha_5 \in \mathbb{C} \right\}
\end{align*}
The claim follows.
\end{proof}

\begin{lemma}
    The space $U^+_{3/2}(4 \cdot 7^2, \chi_{28})$ is the span of
    \begin{align*}
    h&=q + 2q^4 - 3q^9 + 4q^{16} - 5q^{25} - 6q^{36} + O(q^{43})=g_1+2l_1.
\end{align*}
\end{lemma}

\begin{proof}
We compute $\Omega(4\cdot 7^2, \chi_{28})$. First look at all positive integers $t,r$ such that $4tr^2|4\cdot 7^2$. Hence either $t=49$ and $r=1$ or $t=1$ and $r=7$. If $r=1$, then $\phi$, being a character modulo $r$, is the trivial character, contradicting $\phi(-1)=-1$. Hence $r$ can not be $1$. In the second case, we need to find all primitive characters $\phi$ modulo $7$ which map $-1$ to $-1$. Then we have
\begin{itemize}
    \item $1=\phi(-1)^2=\phi(1)$
    \item $\phi(2)^3=\phi(8)=\phi(1)=1$
    \item $\phi(3\cdot 2)=\phi(-1)=-1$ implies that $\phi(3)=-1$
    \item $\phi(3 \cdot 5)=\phi(1)$, we have $\phi(5)=-1$
\end{itemize}
This determines the character and we have $\phi=\left(\frac{-7}{\cdot}\right)$. The multiplicativity of the Kronecker symbol implies that $\chi_{28}=\left(\frac{-1}{\cdot}\right)\phi$ as a character mod $4\cdot7^2$, hence $(\left(\frac{-7}{\cdot}\right), 1)$ satisfies all conditions and we showed that $\Omega(4\cdot 7^2,\chi_{28})=\left\{\left(\left(\frac{-7}{\cdot}\right), 1\right)\right\}$. Thus $U_{3/2}^+(4 \cdot 7^2,\chi_{28})$ is the one dimensional subspace spanned by $h(\left( \frac{-7}{\cdot}\right), z)=\frac{1}{2}\sum_{m \in \Z} \left( \frac{-7}{m}\right)m e^{2 \pi m^2 z}=\sum_{m =0}^{\infty}m \left( \frac{-7}{m}\right) e^{2 \pi m^2 z}$. A computation shows that
\begin{align*}
    h&=q + 2q^4 - 3q^9 + 4q^{16} - 5q^{25} - 6q^{36} + O(q^{43})=g_1+2l_1.
\end{align*}
\end{proof}
We compute the orthogonal complement of $U^+_{3/2}(4 \cdot 7^2, \chi_{28})$.
\begin{lemma}\label{V_{3/2}(196)}
    The space $V_{3/2}^+(4 \cdot 7^2, \chi_{28})$ is spanned by the two functions
\begin{align*}
    g_1&=q-2q^8+q^9-2q^{16}+q^{25}+2q^{29}-2q^{32}+2q^{37}+ O(q^{43})  \text{ and } \\
    f_1&=q^5-2q^{12}-q^{13}-q^{17}+2q^{20}+2q^{24}+2q^{33}- q^{41}+O(q^{43}).
\end{align*}
\end{lemma}

\begin{proof} 
To compute the orthogonal complement of $\langle h \rangle$, we compute that $g_1,f_1$ and $h$ are eigenfunctions of the Hecke operator $T(3^2)$. Since the first one have eigenvalue $0$, but the latter one has eigenvalue $-4$, they are orthogonal. This proves the claim.
\end{proof}

We now apply Proposition \ref{Ueda-proposition} to further get:

\begin{corollary}
    The two functions $g_1$ and $f_1$ have the same system of Hecke eigenvalues as $F$.
\end{corollary}

\begin{proof}
From Lemma \ref{V_{3/2}(196)}, we know that the left hand side in Proposition \ref{Ueda-proposition}, the space $V^+_{3/2}(4 \cdot 49, \chi_{28})$ is equal to the span of $g_1$ and $f_1$. It remains to describe the right hand side. 

First note that $S_2(49)^{new}=S_2(49)$ is a one dimensional subspace spanned by $F$. We compute that $\f|W_{49}=-\f$ and $\f\otimes\rho_p=\f$. Moreover, $S_2(7)^{new}$ and $S_2(1)^{new}$ are trivial. Thus $F \in S_2(49)^{-,-}$, i.e. $S_2(49)^{-,-}=S_2(49)$, and the other three subspaces are trivial: $S_2(49)^{+,+}=S_2(49)^{+,-}=S_2(49)^{-,+}=\{0\}$.
Since $\left(\frac{-1}{7}\right)=-1$, the right hand side of the isomorphism in Proposition \ref{Ueda-proposition} becomes
\begin{align*}
    V^+_{3/2}\left(4\cdot7^2, \chi_{28}\right) \cong 2 S_2(49)^{-,-}=2 S_2(49).
\end{align*}
\end{proof}

However since $g_1$ and $f_1$ have relatively many zero coefficients (as they are in the Kohnen space), applying Waldspurger's Theorem to them only allows us to cover about half of the twists. Thus we need to continue our search for suitable functions. To do so, we will use the Shimura lift in the following section.

\subsection{Shimura Lifts}

The Shimura lift maps half-integral functions to integer weight functions, preserving their Hecke eigenvalues.

\begin{theorem}[Shimura]\label{Shimura}
    Let $k\geq 1$ be an odd integer.
    Let $f= \sum c_n q^n \in M_{k+1/2}(4N,\chi)$ be a half-integral weight modular form. Let $t$ be a squarefree integer. Let
    \begin{align*}
        \chi_t(d)=\chi(d) \left( \frac{(-1)^{k}t}{d}\right)
    \end{align*}
    be a character modulo $4Nt$. Set 
    \begin{align*}
        \Sh_t(f):=\sum_{n=0}^{\infty} \left( \sum_{d | n} \chi_t(d)d^{k-1} c_{t (n/d)^2} \right) q^n.
    \end{align*}
    Then $\Sh_t(f) \in M_{2k}(2N, \chi^2)$.
    Moreover, if $p \nmid 4tN$, then 
    \begin{align*}
        \Sh_t(T(p^2)f)=T(p)\Sh_t(f).
    \end{align*}
\end{theorem}

Before using it as a tool to find more suitable functions, we apply it to the functions found in the previous section.

\paragraph{Shimura-Lift of $g_1$ and $f_1$.}
We study $\Sh_1(g_1), \Sh_3(f_1) \in S_2(98)$. Note that $\Sh_1(f_1)=\Sh_2(f_1)=0$. Neither $g_1$ nor $f_1$ gets lifted to $F$:
\begin{align*}
    \Sh_1(g_1)&= q - 2q^4 - 2q^8 - 3q^9 + 4q^{11} + 2q^{16} + 8q^{23} - 5q^{25} + 2q^{29} + O(q^{30})\\
    \Sh_3(f_1)&=-2q^2 - 2q^4 + 2q^8 + 6q^{16} + 6q^{18} - 8q^{22} + O(q^{30}). 
\end{align*}

However, they get mapped to a linear combination of $\f$ and an old form arising from $\f$: We set $F_{old}:=\sum_{n=0}^{\infty}a_{2n}(\f)q^n$. Its  Fourier extension is:
\begin{align*}
    F_{old}=&q-q^2-3q^4-q^8-3q^9+4q^{11}+5q^{16}+ 3q^{18}-4q^{22}+8q^{23}\\
    &-5q^{25}+2q^{29}+O(q^{30}).
\end{align*}
As proved in equation (\ref{old-same-heigenvalues}), this form has the same system of Hecke eigenvalues as $F$. 

We have that $\Sh_1(g_1)=(1/2)(\f+F_{old})$ and $\Sh_3(f_1)=F_{old}-\f$. By linearity of the Hecke operators, this shows that indeed $\Sh_1(g_1)$ and $\Sh_3(f_1)$ have the same system of Hecke eigenvalues as $\f$.

We use this, to search for more half-integral weight forms satisfying Waldspurger's Theorem. We first computed the basis of spaces of half-integral weight modular forms corresponding to different character with modulus $4\cdot 49$, applied the Shimura lift to these functions and checked if any linear combinations of $F$ and $F_{old}$ could be found in their span.
\begin{lemma}
    The functions
    \begin{align*}
     f_2&=q-q^2-q^4-3q^8+q^9+2q^{15}-3q^{16}+q^{18}-2q^{22}+2q^{23}+q^{25} \\
    &+4q^{29}+2q^{30}-q^{32}-q^{36}+4q^{37}+2q^{39}+O(q^{43}) \text{ and }\\
    f_3&=q^3-2q^5-q^6+3q^{12}+2q^{13}+q^{17}-q^{19}-2q^{20}-3q^{24}+2q^{26}\\
    &-2q^{31}-2q^{33}
    +q^{34}- q^{38}+q^{41}+O(q^{43})
\end{align*}
have the same system of Hecke eigenvalues as $F$ for almost all primes.
\end{lemma}

\begin{proof}
We have that $\Sh_1(f_2)=F_{old}$ and $\Sh_3(f_3)=2\f-F_{old}$. So it only remains to show that these functions are indeed eigenforms for almost all primes. For any squarefree $t$, the Shimura lift is a linear function mapping
\begin{align*}
    \Sh_t: S_{3/2}\left( 4\cdot49, \chi_{28} \right) \rightarrow M_2\left( 2 \cdot 49, \chi_{249} \right).
\end{align*} 
One can compute that $\ker(\Sh_1) \cap \ker(\Sh_3)=\{0\}$ and $\Sh_3(f_2)=0$. Hence for primes $p\neq 2,7$:
$$a_p(F_2)f_2-T(p)f_2 \in \ker(\Sh_1) \cap \ker(\Sh_3).$$ Thus $a_p(F)f_2=T(p)f_2$, i.e. $f_2$ is an eigenform for almost all primes with the same system of eigenvalues as $F$. The same reasoning applies to $f_3$, since $\Sh_1(f_3)=0$.
\end{proof}

The functions $f_1,f_2,f_3$ and $g_1$ are linearly independent.
\begin{remark}
    Computations with the code of Purkait \cite{CodePurkait} confirm that the subspace of $S_{3/2}\left( 4\cdot 49, \chi_{28}\right)$ consisting of the functions with the same system of eigenvalues as $F$ is spanned by these functions.
\end{remark}

\subsection{Theta Functions}
To describe a criterion for when the coefficients of the functions found in the previous sections are zero, we express them in terms of theta functions. 
We use the Theorem 10.9 in \cite{Iwaniec}:
\begin{theorem}\label{theo-matrices-theta-function}
    Let $A$ be a positve definite $r\times r$ matrix with integer entries and even diagonal entries. Let $Q(v):=(1/2)v^TAv$ be the associated quadratic form and set $r_Q(n)=\#\{v \in \Z^r : \quad Q(v)=n\}$.  Define:
    \begin{align*}
        \theta_Q(z)=\sum_{n=0}^{\infty} r_Q(n)q^n.
    \end{align*}
Let $N$ be the smallest positive integer such that $NA^{-1}$ has integer entries and even diagonal entries. Then
\begin{align*}
    \theta_Q(z) \in M_{r/2}(\Gamma_0(N),\chi_{|2A|}).
\end{align*}
\end{theorem}
 We have the following result:

\begin{lemma}\label{theta-matrices}
We can write:
\begin{align*}
    f_1&=\frac{-1}{2}\theta_1+\frac{1}{2}\theta_2,\\
    f_2&=\frac{1}{2}\theta_4+\frac{-1}{2}\theta_7,\\
    f_3&=\frac{1}{2}\theta_{12}+\frac{-1}{2}\theta_{13}  \quad \text{and} \\
    g_1&=\frac{-1}{4}\theta_3+
    \frac{1}{4}\theta_4+
    \frac{1}{4}\theta_5+
    \frac{1}{4}\theta_6+
    \frac{1}{4}\theta_7-
    \theta_8-
    \frac{1}{2}\theta_9+
    \frac{1}{4}\theta_{10}+
    \frac{1}{2}\theta_{11}.\\
\end{align*}
where $\theta_i$ is the theta function associated to the matrix $M_i$.
\begin{align*}
    &M_1:=\begin{pmatrix}
        26 & 8 & -8 \\
        8 & 24 & 4 \\
        -8 & 4 & 24
    \end{pmatrix}, \quad
    &M_2:=\begin{pmatrix}
        34 & 2 & 6 \\
        2 & 10 & 2 \\
        6 & 2 & 34
    \end{pmatrix}, \quad
    &M_3:=\begin{pmatrix}
        34 & 2 & -6 \\
        2 & 4 & 2 \\
        -6 & 2 & 22
    \end{pmatrix}, \\
    &M_4:=\begin{pmatrix}
        56 & 0 & 14 \\
        0 & 2 & 0 \\
        14 & 0 & 28
    \end{pmatrix},\quad
    &M_5:=\begin{pmatrix}
        70 & 14 & -28 \\
        14 & 42 & 14 \\
        -28 & 14 & 70
    \end{pmatrix}, \quad
    &M_6:=\begin{pmatrix}
        98 & 0 & 0 \\
        0 & 2 & 0 \\
        0 & 0 & 14
    \end{pmatrix}, \\
    &M_7:=\begin{pmatrix}
        98 & 0 & 0 \\
        0 & 4 & 2 \\
        0 & 2 & 8
    \end{pmatrix}, \quad
    &M_8:=\begin{pmatrix}
        98 & 0 & 0 \\
        0 & 8 & 4 \\
        0 & 4 & 16
    \end{pmatrix}, \quad
    &M_9:=\begin{pmatrix}
        98 & 0 & 0 \\
        0 & 14 & 7 \\
        0 & 7 & 28
    \end{pmatrix},\quad \\
    &M_{10}:=\begin{pmatrix}
        98 & 0 & 0 \\
        0 & 14 & 0 \\
        0 & 0 & 98
    \end{pmatrix}, \quad
    &M_{11}:=\begin{pmatrix}
        98 & 0 & 0 \\
        0 & 28 & 14 \\
        0 & 14 & 56
    \end{pmatrix}, \quad
    &M_{12}:=\begin{pmatrix}
        24 & 2 & 10\\
        2 & 6 & 2\\
        10 & 2 & 24
    \end{pmatrix}\\
    &M_{13}:=\begin{pmatrix}
        34 & 2 &-2\\
        2 &10 & 4\\
        -2 & 4 & 10
    \end{pmatrix}.
\end{align*}

\end{lemma}
We note that the theta series $\theta_1,...,\theta_{13}$ are linearly independent.

\medskip To compute the theta series of a given matrix, we used the fact that for any symmetric, positive definite $n \times n$ matrix $A$, we have  $v^TAv \geq \lambda_{min} |v|^2$ where $\lambda_{min}=\min \{ \text{Eigenvalues } A \}$, which allowed us to bound the norm (and thus the entries) of the vector we need to consider.
Thus given matrices it is easy to check that they span indeed the functions $f_1,f_2,f_3$ and $g_1$. In the Appendix \ref{Computing Theta-Functions}, we describe how we found suitable matrices. 

\section{Proof of the main Theorem}
In this section we prove Theorem \ref{Main-Theorem}. The proof consists of $6$ steps.
\begin{proof}
\textbf{Step 1:  Reducing the problem to studying L-functions of modular forms.}
Fix a discriminant $d$.
We know that $E_1(\Q(\sqrt{d}))$ has non-torsion points if and only if $E_d(\Q)$ has non-torsion points (see also Lemma \ref{torsion-correspondance}). By the Birch and Swinnerton-Dyer Conjecture (\ref{BSD}) $E_d(\Q)$ has non-torsion points if and only if  $L(E_d,1)=0$. Let $F$ be the modular form associated to $E_1$. Then the modular form associated to $E_d$ is given by $F_d:=F \otimes \chi_d$ (see also equation \ref{eq: modular form of E_d}) and $L(F_d,1)=L(E_d,1)$. In the remainder of the proof, we will therefore study $L(F_d,1)$.

 \textbf{Step 2: Waldspurger's Theorem.} The main tool to determine when $L(F_d,1)$ vanishes is Waldspurger's Theorem (\ref{Waldspurger}). We apply this in our setting: Recall that $F$ is an eigenform and that $f_1,f_2$ and $f_3$ have the same system of Hecke eigenvalues. For the moment assume that $d_1$ and $d_2$ are two square-free positive integers such that $d_1/d_2$ is a non-zero square in $\Q_2$ and $\Q_7$ (we give more details on how to choose $d_1$ and $d_2$ in step 4). For $j \in \{1,2\}$, we now describe the character $\chi_0^{-1}\chi_{d_j}$ from Waldspurger's Theorem. Recall that $\chi_0(d_j):=\chi(d_j) \left( \frac{-1}{d_j}\right)$ where $\chi=\chi_{28}$ is the character of the half-integer weight form. So we get $\chi_0^{-1}\chi_{d_j}=\chi_{-28d_j}$. For $ i \in \{1,2,3\}$ Waldspurger's Theorem implies that:
\begin{align}\label{eq: Waldspurger-for-X0(49)-1}
        c_{d_1}(f_i)^2L(\f\otimes \chi_{-28d_2},1) \chi(d_2/d_1)d_2^{1/2} = c_{d_2}(f_i)^2L(\f\otimes \chi_{-28d_1},1) d_1^{1/2}.
\end{align}

\textbf{Step 3: Rewriting $L( F\otimes \chi_{-28d_j}, 1)$.}
In this step for $j=1,2$ we show that $L( F\otimes \chi_{-28d_j}, 1)$ is essentially $L( F\otimes \chi_{d_j}, 1)$. First note that
\begin{align}\label{eq: L(F_-7d4)}
    L\left(\f\otimes \chi_{-28d_j},1\right)
    =L\left(\f\otimes \chi_{-7} \otimes \chi_{d_j} \otimes \chi_4,1\right).
\end{align}
We claim that $F\otimes \chi_{-7}=\f$. Indeed it suffices to show that whenever $\left( \frac{-7}{d} \right)\neq 1$, then $a_d=0$. Recall, that for primes $p$ the coefficients $F$ are given by the trace of the Frobenius of the elliptic curve $E_1$. By equation (\ref{Frobenius_CM_-1}) for any prime $p \geq 5$ of good reduction, $a_p=0$ if and only if $\left( \frac{-7}{p} \right)=-1$.
    Moreover:
\begin{align*}
    \left( \frac{-7}{2} \right)=1 &\text{ and } a_2=1,\quad
    \left( \frac{-7}{3} \right)=-1 &\text{ and } a_3=0, \text{ and }
    \left( \frac{-7}{7} \right)=0 &\text{ and } a_7=0.
\end{align*}
Thus
\begin{align}\label{eq: F_-7d=F_d}
    \f_{-7d}=\f_d
\end{align}
 and therefore from equation \ref{eq: L(F_-7d4)} it follows that
\begin{align*}
    L\left(\f\otimes \chi_{-28d_j},1\right)
    =L\left(F\otimes \chi_{d_j} \otimes \chi_{4},1\right).
\end{align*}
 By Proposition \ref{prop-twist-L-function} we know that $\f_{d_j}= F\otimes \left(\frac{d_j}{\cdot}\right)$ is an eigenform with character $\chi_{49{d_j}^2}$ and we apply Corollary \ref{L-function-p-twist-1} with $p=2$ to get:
 \begin{align*}
    L\left(\f\otimes \chi_{-28d_j},1\right)
    &=L\left(\f_{d_j},1\right)
    \left(1-\left(a_2(\f_{d_j})-\chi_{49d_j^2}(2)\right)2^{-1}\right).
    \end{align*}
Finally noting that $a_2(\f)=1$ we compute:
\begin{align*}
     C(d_j):&=1-\left(a_2(\f_{d_j})-\chi_{49n^2}(2)\right)2^{-1}
    =1-\left(a_2(\f)\left( \frac{d_j}{2}\right)-\left( \frac{d_j}{2}\right)^2\right)2^{-1}\\
    &=1-\left(\left( \frac{d_j}{2}\right)-\left( \frac{d_j}{2}\right)^2\right)2^{-1}
   =\begin{cases}
        1 & \text{ if }\left( \frac{d_j}{2}\right) \in \{ 0,1\}\\[1.5mm]
        2 &\text{ if }\left( \frac{d_j}{2}\right) =-1
    \end{cases}
\end{align*}
Therefore
 \begin{align*}
    L\left(\f\otimes \chi_{-28d_j},1\right)
    =L\left(\f_{d_j},1\right)C(d_j).
\end{align*}
Thus equation \ref{eq: Waldspurger-for-X0(49)-1} becomes:
\begin{align}\label{eq: Waldspurger-for-X0(49)-2}
        c_{d_1}(f_i)^2C(d_2)L(\f\otimes \chi_{d_2},1) \chi(d_2/d_1)d_2^{1/2} = c_{d_2}(f_i)^2L(\f\otimes \chi_{d_1},1) d_1^{1/2}C(d_1).
\end{align}

\textbf{Step 4: Finding suitable choices for $d_2$ and $f_i$.}
 Set $d_1:=d$ is the discriminant in whose twist we are interested in. For each choice of $d_1$, we want to find a suitable discriminant $d_2$ and $i \in \{1,2,3\}$ such that:
\begin{itemize}
    \item $d_2$ is square-free and $d_1/d_2$ is a non-zero square in $\Q_2$ and $\Q_7$.
    \item $c_{d_2}(f_i)\neq 0$.
    \item $L(F_{d_2},1)\neq 0$
\end{itemize}

 We recall some properties of squares in $p$-adic fields:
    \begin{itemize}
        \item If $d_1=2^{m_1}u_1$ and $d_2=2^{m_2}u_2$ where $u_i$ is a unit in $\Q_2$ then $d_1/d_2 \in (\Q_2^*)^2$ if and only if $m_1\equiv m_2 \mod 2$ and $u_1\equiv u_2 \mod 8$.
        \item For any odd prime $p$, let $d_1=p^{m_1}u_1$ and $d_2=p^{m_2}u_2$ where $u_i$ is a unit in $\Q_p$. Then $d_1/d_2 \in (\Q_p^*)^2$ if and only if $m_1\equiv m_2 \mod 2$ and $u_1\equiv s u_2$ where $s$ is a square modulo $p$.
    \end{itemize}
Since we assume that $(d,7)=1$, we only need to distinguish the case $(d,2)=1$ and $(d,2)=2$. If $d_1:=d$ is odd, then any $d_2$ with $d_2 \equiv d_1 \mod 8$ and $\left( \frac{d_1}{7}\right)=\left( \frac{d_2}{7}\right)$ is suitable. If $d_1$ is even, then we need that $d_1/2 \equiv d_2/2 \mod 8$ and $\left( \frac{d_1}{7}\right)=\left( \frac{d_2}{7}\right)$. 

Computations with MAGMA yield the following:
\medskip
\begin{itemize}
    \item[(i)]
Let $\left(\frac{d_1}{7} \right)=1$. For odd numbers we have:

\begin{center}
    \begin{tabular}{c c c c}
    $d_1 \mod 8$ & $d_2$  &$c_{d_2}(f_2)$& $L\left(\f_{d_2}, 1\right)$\\
    \hline
    1& 1& 1 & 0.9666\\
    3& 51&  -2& 1.0828\\
    5& 29&  4& 0.7180\\
    7& 15&  2& 1.9967\\
    \end{tabular}
\end{center}
For the last column we compute the $L$-functions values $L(F_d,1)$ up to $30$ decimals in MAGMA.
Similarly for even numbers:
\begin{center}
    \begin{tabular}{c c c c}
    $d_1/2 \mod 8$ & $d_2$  &$c_{d_2}(f_2)$& $L\left(\f_{d_2}, 1\right)$\\
    \hline
    1& 2& -1 & 1.3670\\
    3& 22&  -2& 1.6487\\
    5& 58&  -2& 1.0154\\
    7& 30&  2& 1.4118\\
    \end{tabular}
\end{center}

\item[(ii)]
Let $\left( \frac{d_1}{7}\right)=-1$ and $d_1\not\equiv 1 \mod 4$. 
First assume that $d_1$ is odd, then we have:

    \begin{center}
    \begin{tabular}{c c c c}
    $d_1 \mod 8$ & $d_2$  &$c_{d_2}(f_3)$& $L\left(\f_{d_2}, 1\right)$\\
    \hline
    3& 3& 1 & 2.2323\\
    7& 31& -2 & 2.777\\
    \end{tabular}
    \end{center}

Now assume that $d_1$ is even, then

    \begin{center}
    \begin{tabular}{c c c c}
    $d_1/2 \mod 8$ & $d_2$  &$c_{d_2}(f_3)$& $L\left(\f_{d_2}, 1\right)$\\
    \hline
    1& 34& 1& 0.6631\\
    3& 6& -1& 1.5785\\
    5& 26& 2& 3.0332\\
    7& 94& 2& 1.5952\\
    \end{tabular}
    \end{center}

\item[(iii)]
If $\left( \frac{d_1}{7}\right)=-1$ and $d_1\equiv1\mod 4$, then
\begin{center}
    \begin{tabular}{c c c c c c c c c}
    $d_1 \mod 8$ & $d_2$ &    $c_{d_2}(f_1)$ & $L\left(\f_{d_2}, 1\right)\approx$\\
    \hline
    1& 17&  -1 & 0.4688\\
    5& 5&  1&  0.8646 \\
    \end{tabular}
    \end{center}
\end{itemize}

\textbf{Step 5: Rearranging the equation \ref{eq: Waldspurger-for-X0(49)-2} from Waldspurger's Theorem.}
For $d_1:=d$, let $d_2$ and $f_i$ be as in step 4. Note that the choice of $d_2$ forced that $(d_2,7)=1$ and that $(d_1,2)=(d_2,2)$, so $\chi_{28}(d_2/d_1) \neq 0$. Therefore, we may rearrange the equation \ref{eq: Waldspurger-for-X0(49)-2} and get:
\begin{align*}
    L&\left(\f\otimes \chi_{d_1},1\right)\\
    &=(c_{d_1}(f_i)/c_{d_2}(f_i))^2 L\left(\f\otimes \chi_{d_2},1\right) \chi_{28}(d_2/d_1)(d_2/d_1)^{1/2}(C(d_2)/C(d_1))
\end{align*}
By our assumptions from the previous step, the right hand side is zero if and only if $c_{d_1}(f_i)=0$.

\textbf{Step 6: Rewriting the coefficients of $f_i$.}
Recall that from Lemma \ref{theta-matrices}
\begin{align*}
    f_1=\frac{-1}{2}\theta_1+\frac{1}{2}\theta_2, \quad 
    f_2=\frac{1}{2}\theta_4+\frac{-1}{2}\theta_7,\text{ and }
    f_3=\frac{1}{2}\theta_{12}+\frac{-1}{2}\theta_{13}.
\end{align*}
Where for any matrix
\begin{align*}
    M_{i}:=\begin{pmatrix}
        a & b & c \\
        b & d & e \\
        c & e & f
    \end{pmatrix}
    \end{align*}
 as in the statement of the lemma the corresponding theta series is given by:
    \begin{align*}    \theta_{Q_i}(z)=\sum_{n=0}^{\infty} \#\{(x,y,z) \in \Z^3 : \quad (a/2)x^2+bxy+cxz+(d/2) y^2 +e yz + (f/2)z^2=n\} q^n.
\end{align*}
Writing this out with the matrices given in \ref{theta-matrices} yields the conditions in the statement. 
\end{proof}

\begin{remark}
    If $\left( \frac{d_1}{7}\right)=-1$ and $d_1 \equiv 1 \mod 8$, then since also $c_{17}(f_3)=-1$, we can also apply the third criterion of Theorem \ref{Main-Theorem}.

    It is possible to apply Theorem \ref{Main-Theorem} to the function $g_1$, however we did not find a good decomposition of $g_1$ in terms of theta series.
\end{remark}
In the theorem we assume that the discriminant $d$ is positive and coprime to $7$. Moreover in step $3$ of the proof (see equation \ref{eq: F_-7d=F_d}) we show that for any discriminant $d$ we have $F_{-7d}=F_d$. Therefore $E_1$ has non-torsion points over $Q(\sqrt{-7d})$ if and only if $E_1$ has non-torsion points over $Q(\sqrt{d})$. Thus it remains to study the rank of $E_d$ when $d<0$ is not divisible by $7$. We show that in this case $E_d$ has always non-torsion points:
\begin{lemma}
    Let $d<0$ be coprime to $7$. Then assuming the BSD conjecture (\ref{BSD}) $E(\Q(\sqrt{d}))$ is of positive rank.
\end{lemma}
\begin{proof}
Following step $1$ of the proof above, we study the $L$-function of $F\otimes \chi_d$. Recall (see equation (\ref{functional-equation-l-function})) that $\Lambda(F\otimes \chi_d,s)$, the completion of the $L$-function of $F \otimes \chi_d$, satisfies the following functional equation:
\begin{align*}
    \Lambda(F \otimes \chi_d,1)=-w_{49} \chi_d(-49) \Lambda(F \otimes \chi_d, 1)
\end{align*}
where $w_{49}=-1$ is the eigenvalue of $F$ under the Fricke involution. Since:
\begin{align*}
    \chi_d(-49)=\left( \frac{d}{-49}\right)=\left( \frac{d}{-1}\right)=-1.
\end{align*}
we have that $\Lambda(F\otimes \chi_D, k)=0$ and hence $L(F \otimes \chi_D,k)=0$.
\end{proof}
With this lemma we covered all possibilities for the discriminant $d$ and are therefore able to determine the rank of $E_1$ over any quadratic extension $Q(\sqrt{d})$.

\subsection*{Acknowledgements}
The author would like to thank Eugenia Rosu for suggesting this project and for many helpful discussions. 
\appendix
\section{Torsion of the Twists of $X_0(49)$}
Below we compute the torsion of group of the quadratic twists $E_d(\Q)$

\begin{lemma}\label{torsion-E_d(Q)}
    The torsion group of $\E_d(\Q)$ is isomorphic to $\Z/2\Z$ and is generated by $(d7,0)$
\end{lemma}

\begin{proof}
    We first show that we have two-torsion. Let $P=(x,y)$ be a two torsion point satisfying the equation $y^2=x^3-35d^2x-98d^3$. Then $y=0$ and $(x/d,0)$ satisfies $dy^2=x^3-35x-98$. Thus $(x/d,0)$ is a two torsion point on $\E(\Q)$. We already showed that the only value for $x/d=7$, hence $P=(d7,0)$.

    Now assume that $\#\E_d(\Q)_{tors}=N$ and that there exists a prime $p_0>3$ dividing $N$. For any prime $p$ such that $E_1$  and $\E_d$ have good reduction mod $p$ we have that $a_p(E_d)=\left( \frac{d}{p}\right) a_p(E_1)$. In particular we have that $\left( \frac{-7}{p}\right)=-1$ if and only if $\#\E_d(\F_p)=p+1$.

    We use this to find a contradiction: According to Dirichlet's Theorem on arithmetic progressions, there exist infinitely many primes $p$ of the form $p= 3 + 7p_0n$ for some $n \in \N$. Pick $p$ such that $\E_1$ and $\E_d$ have good reduction mod $p$. Using quadratic reciprocity, we get:
    \begin{align*}
        \left( \frac{-7}{p}\right)
        =\left( \frac{p}{-7}\right)
        =\left( \frac{3}{7}\right)=-1.
    \end{align*}
    Thus $a_p(\E_d)=a_p(\E_1)=0$ and we get we get:
    \begin{align*}
    p_0 \mid \#\E_d(\F_p)=p+1=4+7p_0n.
    \end{align*}
    Hence $4 \equiv 0 \mod p_0$. A contradiction.

    Similarly one can also show that $4$ does not divide $N$, by looking at primes of the form $p=5+4\cdot7n$.

    If $3$ divides $N$ then we have a three torsion point $P=(x,y)$. Thus $2P=-P=(x,-y)$. The $x$-coordinate of $2P$ is given by
\begin{align*}
    x=\left(\frac{3x^2-35d^2}{2y}\right)^2-2x \quad \Rightarrow \quad 2^23xy^2=(3x^2-35d^2)^2
\end{align*}
Since the right hand side is a square, $x$ has to be of the form $3x_0$ where $x_0$ is a square.
Moreover  we see that $d^4\equiv 0 \mod 3$, so $d=3d_0$, where $(d_0,3)=1$ since $d$ is squarefree. Plugging this into the equation, we get:
\begin{align*}
    2^23^2x_0y^2=(3^3x_0^2-3^235d_0^2)^2=3^4(3x_0^2-35d_0^2)^2\\
    \Rightarrow 2^2x_0y^2=3^2(3x_0^2-35d_0^2)^2
\end{align*}
The left hand side becomes:
\begin{align*}
    2^2x_0y^2
    &=2^2x_0(3^3x_0^3-3^3 35d_0^2x_0-3^3 98d_0^3)\\
    &=3^32^2x_0(x_0^3-35d_0^2x_0- 98d_0^3).
\end{align*}
And thus
\begin{align*}
    3^32^2x_0(x_0^3-35d_0^2x_0- 98d_0^3)&=3^2(3x_0^2-35d_0^2)^2\\
    \Rightarrow 3\cdot2^2x_0(x_0^3-35d_0^2x_0- 98d_0^3)&=(3x_0^2-35d_0^2)^2.
\end{align*}
Then $ d_0^4\equiv 0 \mod 3$, contradicting $(d_0,3)=1$.
Hence $3$ does not divide $N$.
\end{proof}
The next Lemma proves that instead of studying the elliptic curve over a given quadratic extension, we can study its twist.
\begin{lemma}\label{torsion-correspondance}
    We have have a map $ \tau : \E_1(\Q\sqrt{d}) \rightarrow \E_d(\Q)$ such that for any point $P \in \E_1(\Q\sqrt{d})$ we have: $P \in \E(\Q\sqrt{d})_{tors} \Leftrightarrow \tau (P) \in \E_d(\Q)_{tors}$.
\end{lemma}
\begin{proof}
    Let
    \begin{align*}
        \sigma: \Q(\sqrt{d}) \rightarrow \Q(\sqrt{d}),\quad a+b\sqrt{d} \mapsto a-b\sqrt{d}.
    \end{align*}
    be the non-trivial field automorphism on $\Q(\sqrt{D})$. We can extend this map to $\E(\Q(\sqrt{d}))$ by mapping a point $P=(x,y)$ to $\sigma(P)=(\sigma(x), \sigma(y))$. Since $\sigma$ is a multiplicative and additive map, this map is well-defined. After extending it to the point at infinity by $0 \mapsto 0$, this map becomes a group homomorphism.

    Let $P$ be a point on $\E_1(\Q(\sqrt{d}))$. Set $Q:=P-\sigma(P)$. Then $\sigma(Q)=-Q$. Thus $Q$ is of the form $(a,b\sqrt{d})$ where $a$ and $b$ satisfy $dy^2=x^3-35x-98$. The map $\tau$ maps the point $P$ to $(a,b)$. 

    $\Leftarrow$ Let $P_1 \in \E(\Q(\sqrt{d})) \setminus \{0\}$ such that $Q=P-\sigma(P)=(a,b\sqrt{d})$ corresponds to a torsion point in $\E_d(\Q)$. If $Q=0$, then $P=\sigma(P)$, hence $P \in \E(\Q)$ and thus $P=(7,0)$ is a torsion point (see also equation \ref{eq: E over Q}). Otherwise, we have that $(a,b)=(d7,0)$. Thus $2Q=0$, hence $2P=2\sigma(P)=\sigma(2P)$, which implies that $2 P \in \E(\Q)$. Then either $2P=0$ or $2P=(7,0)$. In any case, $4P=0$, thus $P$ is a torsion point.

    $\Rightarrow $ We now show that any torsion point $P$ gets mapped to a torsion point in $\E_d(\Q)$.
    By Corollary 4 in \cite{Gonz_lez_Jim_nez_2013}, for odd $n$ we have that
    \begin{align*}
     \E(\Q(\sqrt{d}))[n] \simeq \E_d(\Q) [n] \times \E(\Q)[n].
    \end{align*}
    We already showed that the right hand side is trivial, so we have no torsion points of odd order in $\E(\Q(\sqrt{d}))$. Thus we may assume that $P \in \E(\Q(\sqrt{d}))[2^n] \setminus \{0\}$ for some smallest $n \in \N$. Then since $2^nQ=2^nP-\sigma(2^nP)=0$, $Q$ is also a torsion point. Now write $Q=(a,b  \sqrt{d})$. This point gets mapped to $(a,b)$ on $E_d: dy^2=x^3-25x-98$ which corresponds to $\tilde{Q}=(da,d^2b)$ on $\E_d: y^2=x^3-35d^2x-98d^3$. For the $x$-coordinates of $2Q$ and $2 \tilde{Q}$, we have:
    \begin{align*}
        x(2Q)&=\left( \frac{3a^2-35}{2b \sqrt{d}} \right)^2-2a, \quad \text{and}\\
        x(2\Tilde{Q})&=\left( \frac{3d^2a^2-35d^2}{2bd^2} \right)^2-2ad=\left( \frac{3a^2-35}{2b} \right)^2-2ad =d\cdot a(2Q).
    \end{align*}

For the $y$-coordinates, we get:
\begin{align*}
    y(2Q)&=-\frac{(3a^2-35)}{2b\sqrt{d}}\left( \left( \frac{3a^2-35}{2b \sqrt{d}} \right)^2-2a \right) -\frac{a^3-35a-2\cdot98}{2b\sqrt{d}} \quad \text{and}\\
    y(2\Tilde{Q})&=-\frac{(3d^2a^2-35d^2)}{2bd^2}\left(\left( \frac{3a^2-35}{2b} \right)^2-2ad\right)-\frac{a^3d^3-35ad^3-2\cdot98d^3}{2bd^2} \\
    &=-\frac{(3a^2-35)}{2b}\left(\left( \frac{3a^2-35}{2b} \right)^2-2ad\right)-\frac{a^3d-35ad-2\cdot98d}{2b}\\
   & =d^{3/2}y(2Q).
\end{align*}

From these computations it is clear that $2Q$ is again of the form $(\alpha,\beta\sqrt{d})$ with $\alpha, \beta \in \Q$ and that moreover, $(d \alpha,d^2\beta)=2\Tilde{Q}$, i.e. the map $\tau$ behaves nicely under the two multiplication by $2$ map. Hence if $Q=(a,y\sqrt{d})$ is a $2$-torsion point, then so is $\Tilde{Q}$. This shows that torsion points get mapped to torsion points.
\end{proof}
This proof uses very basic techniques. We point out that the map $\tau$ is generalizable for any elliptic curve and that also in the general case, $E(\Q(\sqrt{d}))_{tors}$ gets mapped to $E_d(\Q)_{tors}$ with kernel $E(\Q)_{tors}$. One can show that $\tau$ is a group homomorphism using that $E$ and $E_d$ are isomorphic over $\Q(\sqrt{d})$ and that $E_d(\Q) \subset E_d(\Q\sqrt{d})$.

\section{Variations of Sturm's Bound}\label{Sturm}
Sturm \cite{Sturm} proved that to determine a modular form, it suffices to compute a finite number of coefficients, depending on its weight and level. Similar results were proven by Kumar and Purkait  \cite{Purkait-Kumar-2014} for half-integral weight forms as well as for the Kohnen subspace. Note that they work with a slightly different definition of the Kohnen plus space than Definition \ref{Def-kohnen-space}. Their proof can be easily adapted.

\begin{itemize}
    \item Let $F =\sum_{n=0}^{\infty} a_n q^n \in M_k\left(N,\chi\right)$ be a modular form of integer weight. If $a_n=0$ for all $n \leq \frac{k}{12}[SL_2(\Z) : \Gamma_0(N)]$, then $F=0$.
    \item Let $f=\sum_{n =0}^{\infty} c_nq^n \in M_{k/2}(N, \chi)$. If $c_n=0$ when $n \leq \frac{k}{24}[SL_2(\mathbb{Z}):\Gamma_0(N)]$, then $f=0$.
    \item Let $f=\sum_{n=1}^{\infty} c_n q^n \in S_{k/2}(4N,\chi)$. Assume that $\chi=\left( \frac{N_0}{\cdot}\right)$ for some $N_0|N$ and set $\epsilon:=\left(\frac{-1}{N_0}\right)$. If $c_n=0$ for all $n \equiv 2, (-1)^{(k+1)/2} \epsilon \quad (\text{mod 4})$ with $n \leq \frac{k}{24}\cdot[SL_2(\mathbb{Z}):\Gamma_0(4N)]$, then $c_n=0$ for all $n \equiv 2, (-1)^{(k+1)/2} \epsilon \quad (\text{mod 4})$, i.e. $f \in S^K_{k/2}(4N, \chi)$.
\end{itemize}

\section{Computing Theta-Functions}\label{Computing Theta-Functions}
We describe how to compute matrices satisfying the conditions of Theorem \ref{theo-matrices-theta-function}.
\paragraph{Diagonal Matrices} We start with the easiest case, i.e. we assume that $A$ is a diagonal matrix.
Let
\begin{align*}
A=\begin{pmatrix}
    a & 0 & 0 \\
    0 & d & 0 \\
    0 & 0 & f
\end{pmatrix}.    
\end{align*}
As $A$ has to be positive definite, we know that $a,d,f \in \N_{>0}.$ Let $a=a_1...a_{\alpha}$, $d=d_1...d_{\delta}$ and $f=f_1...f_{\phi}$ be the prime decomposition, with $a_i,d_i,f_i>0$. Then
\begin{align*}
    A^{-1}=\begin{pmatrix}
    \frac{1}{a_1...a_{\alpha}} & 0 & 0 \\
    0 & \frac{1}{d_1...d_{\delta}} & 0 \\
    0 & 0 & \frac{1}{f_1...f_{\phi}}
\end{pmatrix}. 
\end{align*}
We assume that $2^2\cdot 7^2 A^{-1}$ is a matrix with even coefficients, so we have that $a,d,f \in \{2,2\cdot 7, 2\cdot 7^2 \}$, as $a,d,f$ are also even. As $2^2\cdot 7^2$ is the smallest integer satisfying this condition, at least one of $a,d,f$ is equal to $2 \cdot 7^2$. Since we are interested in the quadratic forms associated to these matrices, we can assume without loss of generality that $a=2\cdot 7^2$. The determinant is of the form:
\begin{align*}
    \det A =adf=2\cdot7^2df=2 \cdot 7^2 \cdot 2^{e_1} 7^{e_2}.
\end{align*}
As the determinant is of the form $7\cdot 2 \cdot s^2$, so have that $e_1$ is even and $e_2$ is uneven. Thus
\begin{align*}
    d=2 &\Rightarrow f=2\cdot 7,\\
    d=2\cdot7 &\Rightarrow f\in\{ 2, 2\cdot 7^2\},\\
    d=2\cdot7^2 &\Rightarrow f = 2\cdot 7.\\
\end{align*}
As the quadratic forms generated from the matrices are symmetric with respect to the choice of $d$ and $f$, we get the following two matrices:
\begin{align*}
A_1=\begin{pmatrix}
    2 \cdot 7^2 & 0 & 0 \\
    0 & 2 \cdot 7^2 & 0 \\
    0 & 0 & 2 \cdot 7
\end{pmatrix}\quad \text{and }
A_2=\begin{pmatrix}
    2 \cdot 7^2 & 0 & 0 \\
    0 & 2 & 0 \\
    0 & 0 & 2 \cdot 7
\end{pmatrix}.
\end{align*}
These matrices give rise to the following quadratic forms:
\begin{align*}
    Q_{A_1}(x,y,z)&=7^2x^2+7^2y^2+7z^2,\\
    Q_{A_2}(x,y,z)&=7^2x^2+y^2+7z^2.
\end{align*}
The theorem implies that $\theta_{Q_{A_1}},\theta_{Q_{A_1}} \in M_{3/2}\left(2^2\cdot7^2, \left(\frac{7}{\cdot}\right)\right)$. Thus we need to compute the first $42$ coefficients, i.e.
\begin{align*}
    \# \{v=(x,y,z) \in \Z^3: Q_{A_i}(v)=n\}, \text{ for } n \leq 42.
\end{align*}
For any integer $x>0: 7^2x^2>49>42$. So we only need to look at $(x,y,z) \in \Z^3$ with $x=0$.
For $\theta_{Q_{A_1}}$, we can also assume that $y=0$:
\begin{align*}
    Q_{A_1}(0,0,0)&=0\\
    Q_{A_1}(0,0,1)&=7\\
    Q_{A_1}(0,0,-1)&=7\\
    Q_{A_1}(0,0,2)&=28\\
    Q_{A_1}(0,0,-2)&=28
\end{align*}
As $3^2\cdot7=63>42$, these coefficients will be enough to find the theta function attached to $A_1$.
\begin{align*}
    \theta_{A_1}=1+2q^7+2q^{28}+O(q^{43})
\end{align*}
For $A_2$, looking at $|y| \leq 7$ and $|z| \leq 3$, we get
\begin{align*}
\theta_{A_2}:=1&+2q+2q^4+2q^7+4q^8+2q^9+4q^{11}+6q^{16}+4q^{23}+2q^{25}+2q^{28}\\
&+4q^{29}+8q^{32}+2q^{36}+4q^{37}+O(q^{43}).
\end{align*}

Since no linear combination of $g_1,f_1,f_2$ and $f_3$ is in the span of these forms, we need to look at more general matrices.

\paragraph{The General Case.}  Essentially, we are doing the same as above, i.e. using the positive definiteness to find suitbale matrices and then computing the first $42$ coefficients. Since the computations become quickly long, we wrote an algorithm to check, if a matrix of the form
\begin{align*}
    A=\begin{pmatrix}
        a & b & c \\
        b & d & e \\
        c & e & f
    \end{pmatrix}
\end{align*}
satisfies the properties of Theorem \ref{theo-matrices-theta-function}. Note that by applying the vectors $(1,0,0), (0,1,0)$ and $(0,0,1)$, we can see that $a,d,f>0$.
The diagonal entries of the inverse matrix of $A$ are given by $|A|^{-1}(df-e^2)$, $|A|^{-1}(af-c^2)$ and $|A|^{-1}(ad-b^2)$. Since $A^{-1}$ is positive definite as well, this shows that $df-e^2>0$, $af-c^2>0$ and $ad-b^2>0$. These constraints are enough to write an algorithm looking for suitable matrices.

\begin{remark}
    We first followed the approach as in \cite{Tunnell}, aiming to find functions $h_1 \in S_{1/2}(N_1, \chi_1)$ and $h_2 \in S_{1}(N_2, \chi_2)$  such that $h_1h_2 \in S_{3/2}(4\cdot 7^2, \left( \frac{7}{\cdot}\right))$ is a (linear combination of) $g_1,f_1,f_2,f_3$. Computations with MAGMA showed that this is not possible.
\end{remark}

\end{document}